\documentclass[reqno]{article}
\usepackage{amssymb}
\usepackage{xcolor}
\usepackage{amsmath}
\usepackage{hyperref}
\hypersetup{
    colorlinks=true,
    linkcolor=blue,
    filecolor=magenta,      
    urlcolor=cyan,
}

\newtheorem{theorem}{Theorem}[section]

\newtheorem{proposition}[theorem]{Proposition}
\newtheorem{corollary}[theorem]{Corollary}

\newtheorem{remark}[theorem]{Remark}

\newenvironment{proof}{\noindent\textbf{Proof.}\ }{\hfill$\Box$}

\begin{document}

\title{Ideals with approximate unit in semicrossed products}
\author{Charalampos Magiatis}
\date{\today}

\maketitle 

\begin{abstract}
We characterize the ideals of the semicrossed product $C_0(X)\times_\phi\mathbb Z_+$ with left (resp. right) approximate unit.
\end{abstract}


\section{Introduction and Notation}

The semicrossed product is a non-selfadjoint operator algebra which is constructed from a dynamical system. We recall the construction of the semicrossed product we will consider in this work. Let $X$ be a locally compact Hausdorff space and $\phi:X\rightarrow X$ be a continuous and proper surjection (recall that a map $\phi$ is \emph{proper} if the inverse image $\phi^{-1}(K)$ is compact for every compact $K\subseteq X$). The pair $(X, \phi)$ is called a \emph{dynamical system}. An action of $\mathbb{Z}_+:=\mathbb N\cup \{0\}$ on $C_0(X)$ by isometric $*$-automorphisms $\alpha_n$, $n\in\mathbb{Z}_+$, is obtained by defining $\alpha_n(f)=f\circ\phi^n$. We write the elements of the Banach space $\ell^1(\mathbb Z_+,C_0(X))$ as formal series $A=\sum_{n\in\mathbb Z_+}U^nf_n$ with the norm given by $\|A\|_1=\sum_{n\in\mathbb Z_+}\|f_n\|_{C_0(X)}$.  Multiplication on $\ell^1(\mathbb Z_+,C_0(X))$ is defined by setting
\begin{equation*}
(U^nf)(U^mg)=U^{n+m}(\alpha^m(f)g)\ ,
\end{equation*}
and extending by linearity and continuity. With this multiplication $\ell^1(\mathbb Z_+,C_0(X))$ is a Banach algebra.

The Banach algebra $\ell^1(\mathbb Z_+,C_0(X))$ can be faithfully represented as a (concrete) operator algebra on a Hilbert space. This is achieved by assuming a faithful action of $C_0(X)$ on a Hilbert space $\mathcal{H}_0$. Then we can define a faithful contractive representation $\pi$ of $\ell_1(\mathbb Z_+,C_0(X))$ on the Hilbert space $\mathcal H=\mathcal{H}_0\otimes \ell^2(\mathbb Z_+)$ by defining $\pi(U^nf)$ as
\begin{equation*}
\pi(U^nf)(\xi\otimes e_k)=\alpha^k(f)\xi\otimes e_{k+n}\ .
\end{equation*}
The \emph{semicrossed product} $C_0(X)\times_{\phi}\mathbb Z_+$ is the closure of the image of $\ell^1(\mathbb Z_+,C_0(X))$ in $\mathcal{B(H)}$ in the representation just defined. We will denote  an element $\pi(U^nf)$ of $C_0(X)\times_{\phi}\mathbb Z_+$ by $U^nf$ to simplify the notation. 

For $A=\sum_{n\in\mathbb Z_+}U^nf_n\in \ell^1(\mathbb Z_+,C_0(X))$ we call $f_n\equiv E_n(A)$ the \emph{$n$th Fourier coefficient} of $A$. The maps $E_n:\ell^1(\mathbb Z_+,C_0(X))\rightarrow C_0(X)$ are contractive in the (operator) norm of $C_0(X)\times_{\phi}\mathbb Z_+$, and therefore they extend to contractions $E_n:C_0(X)\times_{\phi}\mathbb Z_+ \rightarrow C_0 (X)$.  An element $A$ of the semicrossed product $C_0(X)\times_{\phi}\mathbb Z_+$ is $0$ if and only if $E_n(A)=0$, for all $n \in \mathbb Z_+$, and thus $A$ is completely determined by its Fourier coefficients. We will denote $A$ by the formal series $A=\sum_{n\in\mathbb Z_+}U^nf_n$, where $f_n=E_n(A)$.  Note however that the series $\sum_{n\in\mathbb Z_+}U^nf_n$ does not in general converge to $A$ \cite[II.9 p.512]{peters}. The \emph{$k$th arithmetic mean} of $A$ is defined to be $\bar A_k=\frac{1}{k+1}\sum_{l=0}^k S_l(A)$, where $S_l(A)=\sum_{n=0}^l U^nf_n$. Then, the sequence $\{\bar A_k\}_{k\in\mathbb{Z}_+}$ is norm convergent to $A$ \cite[Remark p.524]{peters}. We refer to \cite{peters, dkm, dfk} for more information about the semicrossed product.

Let $\{X_n\}_{n=0}^{\infty}$ be a sequence of closed subsets of $X$ satisfying
\begin{align}\label{*}
X_{n+1}\cup\phi(X_{n+1})\subseteq X_n\ , \tag{$*$}
\end{align}
for all $n\in\mathbb N$. Peters proved in \cite{p} that there is a one-to-one correspondence between closed two-sided ideals $\mathcal I\subseteq C_0(X)\times_{\phi}\mathbb Z_+$ and sequences $\{X_n\}_{n=0}^{\infty}$ of closed subsets of $X$ satisfying (\ref{*}), under the additional assumptions that $X$ is metrizable and the dynamical system $(X,\phi)$ contains no periodic points. In fact, the ideal $\mathcal I$ associated with the sequence $\{X_n\}_{n=0}^{\infty}$ is $\mathcal I=\{A\in C_0(X)\times_{\phi}\mathbb Z_+:E_n(A)(X_n)=\{0\}\}$.  We will write this as $\mathcal I\sim\{X_n\}_{n=0}^{\infty}$. Moreover, under the above assumptions, he characterizes the maximal and prime ideals of the semicrossed product $C_0(X)\times_{\phi}\mathbb Z_+$.

Donsig, Katavolos and Manousos obtained in \cite{dkm} a characterization of the Jacobson radical for the semicrossed product $C_0(X)\times_{\phi}\mathbb Z_+$, where $X$ is a locally compact metrisable space and $\phi:X\rightarrow X$ is a continuous and proper surjection. Andreolas, Anoussis and the author characterized in \cite{aam2} the ideal generated by the compact elements and in \cite{aam1} the hypocompact and the scattered radical of the semicrossed product $C_0(X)\times_{\phi}\mathbb Z_+$, where $X$ is a locally compact Hausdorff space and $\phi:X\rightarrow X$ is a homeomorphism. All these ideals are of the form $\mathcal I\sim\{X_n\}_{n=0}^\infty$
for suitable families of closed subsets $\{X_n\}_{n=0}^\infty$.

In the present paper we characterize the closed two-sided ideals $\mathcal I\sim\{X_n\}_{n=0}^{\infty}$ of $C_0(X)\times_\phi\mathbb Z_+$ with left (resp. right) approximate unit. As a consequence, we obtain a complete characterization of ideals with left (resp. right) approximate unit under the additional assumptions that $X$ is metrizable and the dynamical system $(X,\phi)$ contains no periodic points.

Recall that a \emph{left} (resp. \emph{right}) \emph{approximate unit} of a Banach algebra $\mathcal A$ is a net $\{u_\lambda\}_{\lambda\in\Lambda}$ of elements of $\mathcal A$ such that:
\begin{enumerate}
\item for some positive number $r$, $\|u_{\lambda}\|\leq r$ for all $\lambda\in\Lambda$,
\item $\lim u_\lambda a=a$ (resp. $\lim au_\lambda =a$), for all $a\in\mathcal A$, in the norm topology of $\mathcal A$.
\end{enumerate}
A net which is both a left and a right approximate unit of $\mathcal A$ is called an \emph{approximate unit} of $\mathcal A$. A left (resp. right) approximate unit $\{u_\lambda\}_{\lambda\in\Lambda}$ that satisfies $\|u_{\lambda}\|\leq 1$ for all $\lambda\in\Lambda$ is called a \emph{contractive left} (resp. \emph{right}) \emph{approximate unit}.

We will say that an ideal $\mathcal I$ of a Banach algebra $\mathcal A$ has a left (resp. right) approximate unit if it has a left (resp. right) approximate unit as an algebra.


\section{Ideals with approximate unit}

In the following theorem the ideals $\mathcal I\sim\{X_n\}_{n=0}^{\infty}$ with right approximate unit are characterized.


\begin{theorem}\label{rau}
Let $\mathcal I\sim\{X_n\}_{n=0}^{\infty}$ be a non-zero ideal of $C_0(X)\times_\phi\mathbb Z_+$. The following are equivalent:
\begin{enumerate}
\item $\mathcal I$ has a right approximate unit.

\item $X_n=X_{n+1}$, for all $n\in\mathbb Z_+$.
\end{enumerate}
\end{theorem}
\begin{proof}
We start by proving that (1) $\Rightarrow$ (2). Let $\mathcal I\sim\{X_n\}_{n=0}^{\infty}$ be an ideal with right approximate unit $\{V_{\lambda}\}_{\lambda\in\Lambda}$. We suppose that there exists $n\in\mathbb Z_+$ such that $X_{n+1}\subsetneq X_{n}$. Let
\begin{equation*}
n_0=\min\{n\in\mathbb Z_+:X_{n+1}\subsetneq X_{n}\}\ ,
\end{equation*}
$x_0\in X_{n_0}\setminus X_{n_0+1}$ and $f\in C_0(X)$ such that $f(x_0)=1$, $f(X_{n_0+1})=\{0\}$ and $\|f\|=1$. Then, for $A=U^{n_0+1}f$, we have
\begin{equation*}
\|AV_{\lambda}-A\| \ge \|E_{n_0+1}(AV_{\lambda}-A)\| = \|fE_0(V_{\lambda})-f\|\ge|(fE_0(V_{\lambda})-f)(x_0)|=1\ ,
\end{equation*}
for all $\lambda\in\Lambda$, since $x_0\in X_{n_0}$ and $E_0(V_{\lambda})(X_{n_0})=0$, which is a contradiction. Therefore $X_n=X_{n+1}$ for all $n\in\mathbb Z_+$.

For (2) $\Rightarrow$ (1), assume that $X_n=X_{n+1}$ for all $n\in\mathbb Z_+$. By \eqref{*}, we get that $\phi(X_0)\subseteq X_0$. We will show that if $\{u_{\lambda}\}_{\lambda\in\Lambda}$ is a contractive approximate unit of the ideal $C_0(X\setminus X_0)$ of $C_0(X)$, then $\{U^0u_{\lambda}\}_{\lambda\in\Lambda}$ is a right approximate unit of $\mathcal I$. Since $\|u_{\lambda}\|\leq 1$, we have $\|U^0u_{\lambda}\|\leq 1$.

Let $A\in\mathcal I$ and $\varepsilon>0$. Then there exists $k\in\mathbb Z_+$ such that
\begin{equation*}
\|A-\bar A_k\|<\frac{\varepsilon}{4}\ ,
\end{equation*}
where $\bar A_k$ is the $k$th arithmetic mean of $A$. Since $X_n=X_0$, $E_n(\bar A_k)\in C_0(X\setminus X_0)$ and $\{u_{\lambda}\}_{\lambda\in\Lambda}$ is an approximate unit of $C_0(X\setminus X_0)$, there exists $\lambda_0\in\Lambda$ such that
\begin{equation*}
\|E_l(\bar A_k)u_\lambda-E_l(\bar A_k)\|<\frac{\varepsilon}{2(k+1)}\ ,
\end{equation*}
for all $l\leq k$ and $\lambda> \lambda_0$. So, for $\lambda>\lambda_0$ we get that
\begin{eqnarray*}
\|AU^0u_\lambda-A\| & < & \|AU^0u_\lambda-\bar A_kU^0u_\lambda+\bar A_kU^0u_\lambda-\bar A_k+\bar A_k-A\| \\
& < & \|AU^0u_\lambda-\bar A_kU^0u_\lambda\|+\|\bar A_kU^0u_\lambda-\bar A_k\|+\|A-\bar A_k\| \\
& < & \|\bar A_kU^0u_\lambda-\bar A_k\|+\frac{\varepsilon}{2}\\
&\leq & \sum_{l=0}^k\|E_l(\bar A_k)u_\lambda -E_l(\bar A_k)\|+\frac{\varepsilon}{2}\\
&<& \varepsilon\ ,
\end{eqnarray*}
which concludes the proof.
\end{proof}


In the following theorem the ideals $\mathcal I\sim\{X_n\}_{n=0}^{\infty}$ with left approximate unit are characterized.


\begin{theorem}\label{lau}
Let $\mathcal I\sim\{X_n\}_{n=0}^{\infty}$ be a non-zero ideal of $C_0(X)\times_\phi\mathbb Z_+$. The following are equivalent:
\begin{enumerate}
\item $\mathcal I$ has a left approximate unit.

\item $X_0\subsetneq X$ and $\phi^{n}(X\setminus X_n)= X\setminus X_0$, for all $n\in\mathbb Z_+$.

\item $\phi(X\setminus X_{1})= X\setminus X_{0}$ and $\phi(X_{n+1}\setminus X_{n+2})=X_{n}\setminus X_{n+1}$, for all $n\in\mathbb Z_+$.
\end{enumerate}
\end{theorem}
\begin{proof}
We start by proving that (1) $\Rightarrow$ (2). Let $\mathcal I\sim\{X_n\}_{n=0}^{\infty}$ be an ideal with left approximate unit $\{V_{\lambda}\}_{\lambda\in\Lambda}$. 

First we prove that $X_0\subsetneq X$. We suppose that $X_0=X$. Then $E_0(V_\lambda)=0$, for all $\lambda\in\Lambda$, and hence for every $U^nf\in\mathcal I$ we have
\begin{equation*}
\|V_\lambda U^nf-U^nf\|\ge\|E_n(V_\lambda U^nf-U^nf)\|=\|E_0(V_\lambda)f-f\|=\|f\|\ ,
\end{equation*}
for all $\lambda\in\Lambda$, which is a contradiction. Therefore $X_0\subsetneq X$.

Now we prove that $\phi^{n}(X\setminus X_n)= X\setminus X_0$, for all $n\in\mathbb Z_+$. We suppose that there exists $n\in\mathbb Z_+$ such that $\phi^{n}(X\setminus X_n)\not\subseteq X\setminus X_0$ and let
\begin{equation*}
n_0=\min\{n\in\mathbb Z_+:\phi^{n}(X\setminus X_n)\not\subseteq X\setminus X_0\}\ .
\end{equation*}
Then, there exist $x_0\in X\setminus X_{n_0}$ such that $\phi^{n_0}(x_0)\in X_{0}$ and a function $f\in C_0(X)$ such that $f(x_0)=1$, $f(X_{n_0})=\{0\}$ and $\|f\|=1$. If $A=U^{n_0}f$, we have that $A\in\mathcal I$, $\|A\|=1$ and
\begin{eqnarray*}
\|V_{\lambda}A-A\| & \ge & \|E_{n_0}(V_{\lambda}A-A)\|\\
& = & \|E_0(V_{\lambda})\circ\phi^{n_0}f-f\|\\
&\ge&|(E_0(V_{\lambda})\circ\phi^{n_0}f-f)(x_0)|\\
&=&1\ ,
\end{eqnarray*}
for all $\lambda\in\Lambda$, since $\phi^{n_0}(x_0)\in X_0$ and $E_0(V_{\lambda})(X_0)=\{0\}$, which is a contradiction. Therefore $\phi^{n}(X\setminus X_n)\subseteq X\setminus X_0$. Furthermore, by \eqref{*} we get that $\phi^n(X_n)\subseteq X_0$, for all $n\in\mathbb Z_+$, and hence
\begin{equation*}
X  =  \phi^n(X)
=\phi^n(X_n\cup(X\setminus X_n))
 =  \phi^n(X_n)\cup\phi^n(X\setminus X_n)
\subseteq X_0\cup \phi^n(X\setminus X_n)\ .
\end{equation*}
Since $\phi^{n}(X\setminus X_n)\subseteq X\setminus X_0$ and $\phi$ is surjective, $\phi^n(X\setminus X_n)=X\setminus X_0$, for all $n\in\mathbb Z_+$.

For (2) $\Rightarrow$ (1), assume that $X_0\subsetneq X$ and $\phi^{n}(X\setminus X_n)= X\setminus X_0$, for all $n\in\mathbb Z_+$. We will show that if $\{u_{\lambda}\}_{\lambda\in\Lambda}$ is a contractive approximate unit of the ideal $C_0(X\setminus X_0)$ of $C_0(X)$, then $\{U^0u_{\lambda}\}_{\lambda\in\Lambda}$ is a left approximate unit of $\mathcal I$. Since $\|u_{\lambda}\|\leq 1$, we have $\|U^0u_{\lambda}\|\leq 1$.

Let $A$ be a norm-one element of $\mathcal I$ and $\varepsilon>0$. Then there exists $k\in\mathbb Z_+$ such that
\begin{equation*}
\|A-\bar A_k\|<\frac{\varepsilon}{4}\ ,
\end{equation*}
where $\bar A_k$ is the $k$th arithmetic mean of $A$. For $l\leq k$, let 
\begin{equation*}
D_\varepsilon(E_l(\bar A_k))=\left\{x\in X: |E_l(\bar A_k)(x)|\ge\frac{\varepsilon}{4(k+1)} \right\}\ .
\end{equation*}
Since $A\in\mathcal I$, we have $E_l(\bar A_k)(X_l)=\{0\}$ and hence $D_\varepsilon(E_l(\bar A_k))\subseteq X\setminus X_l$. Furthermore, since $\phi^{n}(X\setminus X_n)= X\setminus X_0$, for all $n\in\mathbb Z_+$, we have that $\phi^{l}(D_\varepsilon(E_l(\bar A_k)))\subseteq X\setminus X_0$. Moreover, the set $D_\varepsilon(E_l(\bar A_k))$ is compact and hence the set $\phi^{l}(D_\varepsilon(E_l(\bar A_k)))$ is also compact. By Urysohn's lemma, there is a norm-one function $v_l\in C_0(X)$ such that
\begin{equation*}
v_l(x)=\left\{\begin{tabular}{ll} $1$, & $x\in \phi^l(D_\varepsilon(E_l(\bar A_k)))$\\ $0$, & $x\in X_0$ \end{tabular}\right..
\end{equation*}
Then, there exists $\lambda_0\in\Lambda$ such that
\begin{equation*}
\|u_\lambda v_l-v_l\|<\frac{\varepsilon}{2(k+1)}\ ,
\end{equation*}
for all $l\leq k$ and $\lambda>\lambda_0$, and hence
\begin{equation*}
|u_\lambda(x)-1|<\frac{\varepsilon}{2(k+1)}\ ,
\end{equation*}
for all $x\in\cup_{l=0}^k\phi^{l}(D_\varepsilon(E_l(\bar A_k)))$ and $\lambda>\lambda_0$. Therefore, if $x\in\cup_{l=0}^k(D_\varepsilon(E_l(\bar A_k)))$ then $\phi^l(x)\in\cup_{l=0}^k\phi^{l}(D_\varepsilon(E_l(\bar A_k)))$ and hence
\begin{equation*}
\|((u_\lambda\circ\phi^l) E_l(\bar A_k)-E_l(\bar A_k))(x)\|<\frac{\varepsilon}{2(k+1)}\ ,
\end{equation*}
for all $l\leq k$ and $\lambda> \lambda_0$. On the other hand, if $x\not\in\cup_{l=0}^k(D_\varepsilon(E_l(\bar A_k)))$, then
\begin{equation*}
|E_l(\bar A_k)(x)|<\frac{\varepsilon}{4(k+1)}\ ,
\end{equation*}
for all $l\leq k$, and hence 
\begin{equation*}
\|((u_\lambda\circ\phi^l) E_l(\bar A_k)-E_l(\bar A_k))(x)\|<\frac{\varepsilon}{2(k+1)}\ .
\end{equation*}

From what we said so far we get that
\begin{eqnarray*}
\|U^0u_\lambda A-A\| & < & \|U^0u_\lambda\bar A_k-\bar A_k\|+\frac{\varepsilon}{2}\\
& \leq & \sum_{l=0}^k\|(u_\lambda\circ\phi^l)E_l(\bar A_k)-E_l(\bar A_k)\|+\frac{\varepsilon}{2}\\
&<& \varepsilon\ ,
\end{eqnarray*}
for all $\lambda>\lambda_0$.

Now we show that (2) $\Rightarrow$ (3). We assume that $\phi^{n}(X\setminus X_{n})= X\setminus X_0$, for all $n\in\mathbb Z_+$. Then $\phi(X\setminus X_{n+2})\subseteq X\setminus X_{n+1}$. Indeed, if $x\in X\setminus X_{n+2}$ and $\phi(x)\in X_{n+1}$ then $\phi^{n+2}(x)\in X_{0}$, by \eqref{*}, which is a contradiction. Furthermore, by \eqref{*}, we know that $\phi(X_{n+1})\subseteq X_{n}$ and hence $\phi(X_{n+1}\setminus X_{n+2})\subseteq X_n\setminus X_{n+1}$ for all $n\in\mathbb Z_+$.

To prove that $\phi(X_{n+1}\setminus X_{n+2})= X_n\setminus X_{n+1}$ for all $n\in\mathbb Z_+$, we suppose that there exists $n\in\mathbb Z_+$ such that $\phi(X_{n+1}\setminus X_{n+2})\subsetneq X_n\setminus X_{n+1}$. If
\begin{equation*}
n_0 = \min\{n\in\mathbb Z_+:\phi(X_{n+1}\setminus X_{n+2})\subsetneq X_n\setminus X_{n+1}\}\ ,
\end{equation*}
then
\begin{eqnarray*}
\phi(X) & = & \phi(X_{n_0+1}\cup(X\setminus X_{n_0+1}))\\
&=& \phi(X_{n_0+1})\cup\phi(X\setminus X_{n_0+1})\\
& \subseteq & \phi(X_{n_0+1})\cup (X\setminus X_{n_0})\\
&\subsetneq& X\ ,
\end{eqnarray*}
which is a contradiction, since $\phi$ is surjective. Therefore, $\phi (X_{n+1}\setminus X_{n+2})= X_n\setminus X_{n+1}$ for all $n\in\mathbb Z_+$.

Finally we show that (3) $\Rightarrow$ (2). We assume that $\phi(X\setminus X_{1})= X\setminus X_{0}$ and $\phi(X_{n+1}\setminus X_{n+2})=X_{n}\setminus X_{n+1}$, for all $n\in\mathbb Z_+$. Then, $X_0\subsetneq X$. Indeed, if $X_0=X$, then $\mathcal I\equiv\{0\}$ which is a contradiction. If $n>1$, we have that
\begin{eqnarray*}
\phi(X\setminus X_n)&=&\phi\left[(X\setminus X_1)\cup(X_1\setminus X_2)\cup\dots\cup(X_{n-1}\setminus X_n)\right]\\
&=&\phi(X\setminus X_1)\cup\phi(X_1\setminus X_2)\cup\dots\cup\phi(X_{n-1}\setminus X_n)\\
&=&(X\setminus X_0)\cup(X_0\setminus X_1)\cup\dots\cup(X_{n-2}\setminus X_{n-1})\\
&=&X\setminus X_{n-1} \ ,
\end{eqnarray*}
and hence, $\phi^n(X\setminus X_n)=X\setminus X_{0}$, for all $n\in\mathbb Z_+$.
\end{proof}


By Theorem \ref{lau}, if $\mathcal I\sim\{X_n\}_{n=0}^{\infty}$ is an ideal of $C_0(X)\times_\phi\mathbb Z_+$ with left approximate unit, then $X_{n+1}= X_n$ or $X_{n+1}\subsetneq X_n$ for all $n\in\mathbb Z_+$. If $\mathcal I\sim\{X_n\}_{n=0}^{\infty}$ and $X_{n+1}= X_n$, for all $n\in\mathbb Z_+$, we will write $\mathcal I\sim\{X_0\}$. We obtain the following characterization.


\begin{corollary}\label{lau2}
Let $\mathcal I\sim \{X_0\}$ be a non-zero ideal of $C_0(X)\times_\phi\mathbb Z_+$. The following are equivalent:
\begin{enumerate}
\item $\mathcal I$ has a left approximate unit.

\item $\phi(X_0)= X_0$ and $\phi(X\setminus X_0)= X\setminus X_0$.
\end{enumerate}
\end{corollary}
\begin{proof}
By Theorem \ref{lau} we have $\phi(X\setminus X_0)= X\setminus X_0$. By (\ref{*}) we have $\phi(X_0)\subseteq X_0$ and since $\phi$ is surjective we get $\phi(X_0)= X_0$.
\end{proof}


In the following proposition the ideals $\mathcal I\sim\{X_n\}_{n=1}^\infty$ of $C_0(X)\times_\phi\mathbb Z_+$ with left approximate unit are characterized, when $\phi$ is a homeomorphism.


\begin{proposition}
Let $\mathcal I\sim\{X_n\}_{n=1}^\infty$ be a non-zero ideal of $C_0(X)\times_\phi\mathbb Z_+$, where $\phi$ is a homeomorphism. The following are equivalent:
\begin{enumerate}
\item $\mathcal I$ has a left approximate unit.

\item There exist $S,W\subsetneq X$ such that $S$ is closed and $\phi(S)=S$, the sets $\phi^{-1}(W),\phi^{-2}(W),\dots$ are pairwise disjoint and $\phi^k(W)\cap S=\emptyset$, for all $k\in\mathbb Z$, and
\begin{equation*}
X_n=S\cup(\cup_{k=n}^{\infty}\phi^{-k}(W))\ ,
\end{equation*}
for all $n\in\mathbb Z_+$.
\end{enumerate}
\end{proposition}
\begin{proof}
The second condition implies the second condition of Theorem \ref{lau} and hence the implication (2) $\Rightarrow$ (1) is immediate. We will prove the implication (1) $\Rightarrow$ (2).

We set $S=\cap_{n=0}^\infty X_n$. Clearly the set $S$ is closed and, by (\ref{*}), we have $\phi(S)\subseteq S$. We will prove that $\phi(S)=S$. We suppose $\phi(S)\subsetneq S$. Since $\phi$ is surjective, there exists $x\in X\setminus S$ such that $\phi(x)\in S$. Moreover, $\phi^n(x)\in S$ for all $n\ge 1$. However, since $x\notin S$ there exists $n_0$ such that $x\notin X_{n_0}$ and hence $\phi^{n_0}(x)\in X\setminus X_0$, by Theorem \ref{lau}, which is a contradiction since $S\cap (X\setminus X_0)=\emptyset$. 

By Theorem \ref{lau}, $\phi (X_{n+1}\setminus X_{n+2})= X_n\setminus X_{n+1}$ for all $n\in\mathbb Z_+$ and hence $\phi^n (X_{n}\setminus X_{n+1})= X_0\setminus X_{1}$ or equivalently $ X_n\setminus X_{n+1}=\phi^{-n} (X_{0}\setminus X_{1})$ since $\phi$ is a homeomorphism. Furthermore, the sets $\phi^{-1}(X_0\setminus X_1),\phi^{-2}(X_0\setminus X_1),\dots$ are pairwise disjoint.

We set $W=X_0\setminus X_1$. Clearly, $\phi^k(W)\cap S=\emptyset$ for all $k\in\mathbb Z$, since $\phi(S)=S$ and $\phi(W)\subseteq X\setminus X_0$. Also, $X_0=S\cup(X_0\setminus X_1)\cup (X_1\setminus X_2)\cup\dots$ and hence
\begin{equation*}
X_0=S\cup(\cup_{k=0}^{\infty}\phi^{-k}(W))\ .
\end{equation*}
Finally, for all $n\in\mathbb Z_+$ we have that
\begin{equation*}
X_0=X_n\cup(\cup_{k=1}^n X_{k-1}\setminus X_k) = X_n\cup(\cup_{k=1}^n \phi^{-k+1}(W))= X_n\cup(\cup_{k=0}^{n-1} \phi^{-k}(W))\ ,
\end{equation*}
and so
\begin{equation*}
X_n=X_0\setminus(\cup_{k=0}^{n-1} \phi^{-k}(W)) =S\cup(\cup_{k=n}^{\infty} \phi^{-k}(W))\ .
\end{equation*}
\end{proof}


In the following corollary the ideals with an approximate unit are characterized.


\begin{corollary}\label{au}
Let $\mathcal I\sim\{X_n\}_{n=0}^{\infty}$ be a non-zero ideal of $C_0(X)\times_\phi\mathbb Z_+$. The following are equivalent:
\begin{enumerate}
\item $\mathcal I$ has an approximate unit.
\item $X_n=X_{n+1}$, for all $n\in\mathbb Z_+$, and $\phi(X\setminus X_0)= X\setminus X_0$.
\end{enumerate}
\end{corollary}
\begin{proof}
(1) $\Rightarrow$ (2) is immediate from Theorem \ref{rau} and Corollary \ref{lau2}. 

We show (2) $\Rightarrow$ (1). If $X_n=X_{n+1}$, by (\ref{*}), we have $\phi(X_0)\subseteq X_0$. Since $\phi(X\setminus X_0)= X\setminus X_0$ and $\phi$ surjective we have $\phi(X_0)= X_0$. Theorem \ref{rau} and Corollary \ref{lau2} conclude the proof.
\end{proof}


\begin{remark}
If $\mathcal I\sim\{X_n\}_{n=0}^{\infty}$ is an ideal of $C_0(X)\times_\phi\mathbb Z_+$ with a left (resp. right) approximate unit, then it has a contractive left (resp. right) approximate unit. Moreover, the semicrossed product $C_0(X)\times_\phi\mathbb Z_+$ has a contractive approximate unit.
\end{remark}


Let $B$ be a Banach space and $C$ be a subspace of $B$. The set of linear functionals that vanish on a subspace $C$ of $B$ is called the \emph{annihilator} of $C$. A subspace $C$ of a Banach space $B$ is an \emph{$M$-ideal} in $B$ if its annihilator is the kernel of a projection $P$ on $B^*$ such that $\|y\|=\|P(y)\|+\|y-P(y)\|$, for all $y$, where $B^*$ is the dual space of $B$.

Effros and Ruan proved that the $M$-ideals in a unital operator algebra are the closed two-sided ideals with an approximate unit, \cite[Theorem 2.2]{er}. Therefore, we obtain the following corollary about the $M$-ideals of a semicrossed product. 


\begin{corollary}\label{cau}
Let $\mathcal I\sim\{X_n\}_{n=0}^{\infty}$ be a non-zero ideal of $C_0(X)\times_\phi\mathbb Z_+$, where $X$ is compact. The following are equivalent:
\begin{enumerate}
\item $\mathcal I$ is $M$-ideal.
\item $\mathcal I$ has an approximate unit.
\item $X_n=X_{n+1}$, for all $n\in\mathbb Z_+$, and $\phi(X\setminus X_0)= X\setminus X_0$.
\end{enumerate}
\end{corollary}



\bibliographystyle{amsplain}

\end{document}